\documentclass[11pt,reqno]{amsart}
\usepackage{amsmath,amsfonts,amssymb,amsthm,stmaryrd,enumerate,url}
\usepackage{microtype}
\usepackage{bbm,mathrsfs}
\usepackage[utf8]{inputenc}  
\usepackage[english]{babel}
\usepackage[numbers,sort]{natbib}
\usepackage[usenames,dvipsnames,svgnames,table]{xcolor}
\usepackage[colorlinks=true, pdfstartview=FitV, pagebackref=false]{hyperref}
\usepackage{xcolor,tikz}
\usetikzlibrary{trees,arrows,positioning,automata,shadows,fit,shapes}
\usepackage{etoolbox}

\usepackage{shortcuts}

\newcommand{\Dom}{\mathrm{Dom}}
\newcommand{\Ran}{\mathrm{Ran}}
\newcommand{\w}{\mathbf{w}}
\newcommand{\wmin}{\w_{\rm min}}

\usepackage[a4paper,vmargin=3.5cm,inner=2.8cm,outer=2.8cm]{geometry}

\newtheorem{theorem}{Theorem}
\newtheorem{lemma}[theorem]{Lemma}
\newtheorem{proposition}[theorem]{Proposition}
\newtheorem{corollary}[theorem]{Corollary}

\begin{document}

\title[ ]{Cutoff for permuted Markov chains}

\author[ ]{Anna Ben-Hamou}
\author[ ]{Yuval Peres}

\begin{abstract}
Let $P$ be a bistochastic matrix of size $n$, and let $\Pi$ be a permutation matrix of size $n$. In this paper, we are interested in the mixing time of the Markov chain whose transition matrix is given by $Q=P\Pi$. In other words, the chain alternates between random steps governed by $P$ and deterministic steps governed by $\Pi$. We show that if the permutation $\Pi$ is chosen uniformly at random, then under mild assumptions on $P$, with high probability, the chain $Q$ exhibits cutoff at time $\frac{\log n}{\bh}$, where $\bh$ is the entropic rate of~$P$. Moreover, for deterministic permutations, we show that the upper bound on the mixing time obtained by~\citet*{chatterjee2020speeding} may be improved using a result of~\citet*{chen2013comparison} that allows to do away with the dependence on the laziness parameter.
\end{abstract}

\maketitle
\section{Introduction}

\subsection{Setting and main results}

Let $P$ be a bistochastic transition matrix on a finite state space $\Omega$ of size $n$, and let $\pi$ be a permutation on $\Omega$. Denote by $\Pi$ the corresponding permutation matrix (defined by $\Pi_{i,j}=\1_{\{j=\pi(i)\}}$). We are interested in the mixing time of the Markov chain with transition matrix $Q=P\Pi$. By assumption, the uniform distribution on $\Omega$ is stationary for $P$, and it is stationary for $Q$ as well. The total-variation distance to equilibrium at time $t\geq 1$ starting from state $x\in\Omega$ is then given by
\[
\cD_x(t)=\sum_{y\in\Omega} \left( \frac{1}{n}-Q^t(x,y)\right)_+\, .
\]
For $\varepsilon\in (0,1)$, the $\varepsilon$-mixing time is defined as
\[
\tmix(\varepsilon)=\min\{ t\geq 0,\, \cD(t)\leq \varepsilon\}\, ,
\]
where $\cD(t)=\max_{x\in\Omega} \cD_x(t)$. The original Markov chain $P$ might have a very large mixing time (it might not even mix at all, if $P$ is reducible or periodic), and it is natural to ask which permutations on the state space are able to speed up convergence to stationarity. Consider for instance the lazy random walk on the circle $\bbZ_n$ given by $X_0=0$ and, for $k\geq 0$,
\[
X_{k+1}=X_k+\varepsilon_{k+1} \qquad (\mathrm{mod} \; n)\, ,
\]
where $(\varepsilon_k)_{k\geq 1}$ is an i.d.d.\ sequence uniformly distributed on $\{-1,0,1\}$. It is well known that this chain mixes in order $n^2$ steps (and does not have cutoff). Now consider alternating random uniform jumps in $\{-1,0,1\}$ and deterministic jumps given by some bijection $f$ on $\bbZ_n$. This corresponds to the permuted Markov chain
\[
X_{k+1}=f(X_k)+\varepsilon_{k+1} \qquad (\mathrm{mod} \; n)\, .
\]
\citet*{he2021mixing} recently proved that, for $n$ prime and for $f(x)=x^{-1}$ when $x\neq 0$, and $f(0)=0$, this chain takes only order $\log n$ steps to mix. For $f(x)=2x$, it was shown by \citet*{eberhard2021mixing} that the chain has cutoff at time $c\log n$, for some absolute constant $c>0$.

In the general setting, \citet*{chatterjee2020speeding} recently showed that, under some assumptions on $P$, if the permutation $\pi$ satisfies some expansion condition with respect to $P$, then the mixing time of $Q=P\Pi$ is logarithmic, and that all but a vanishing fraction of permutations satisfy this expansion condition.

We refine this result by showing that, under mild assumptions on the matrix $P$, for almost every permutation $\pi$, the chain $Q=P\Pi$ has cutoff at a time $\frac{\log n}{\bh}$, where $\bh$ is the entropic rate of $P$, that is
\begin{equation}\label{eq:def-h}
\bh=\frac{1}{n}\sum_{x,y\in\Omega}P(x,y)\log\frac{1}{P(x,y)}\,  .
\end{equation}
To be more precise, this result holds for a sequence of chains $Q_n=P_n\Pi_n$ on $\Omega_n$ with $|\Omega_n|=n$, all asymptotic statements being understood as $n$ tends to $+\infty$. To alleviate notation, we will later keep this dependence in $n$ implicit.

\begin{theorem}\label{thm:cutoff-random}
Let $P_n$ be a bistochastic transition matrix on $\Omega_n$ with $|\Omega_n|=n\geq 1$, and let $\bh_n$ be the corresponding entropy rate. Assume that
\begin{equation}\label{assum:sparse}
\delta_n=\min\left\{ P_n(x,y) ,\; x,y\in\Omega_n , P_n(x,y)>0\right\}  =\frac{1}{n^{o(1)}}\, ,
\end{equation}

\begin{equation}\label{assum:branching}
\Delta=\sup_{n\geq 1}\max_{x,y\in\Omega_n}P_n(x,y)<1\, ,
\end{equation}

\begin{equation}\label{assum:loglog}
\bh_n=o\left(\frac{\log n}{\log\log n}\right)\, .
\end{equation}
Let $\pi_n$ be a uniform random permutation on $\Omega_n$, and $\Pi_n$ be the corresponding permutation matrix. Then, for all $\varepsilon\in (0,1)$, the mixing time of the chain $Q_n=P_n\Pi_n$ satisfies
\[
\frac{\bh_n\tmix^{(n)}(\varepsilon)}{\log n} \, \underset{n\to +\infty}{\overset{\P}{\longrightarrow}}\,  1 \, .
\]
\end{theorem}

Note that assumption~\ref{assum:sparse} means that $\delta_n$  satisfy  $\frac{\log \delta_n}{\log n} \to 0$. In particular, it holds if $\inf_n \delta_n >0$. Also note that assumption~\ref{assum:branching} implies that $\bh_n\geq \log \frac{1}{\Delta}$. Hence the cutoff time $\frac{\log n}{\bh_n}$ is always $O(\log n)$ but it may be much smaller than $\log n$: assumption~\ref{assum:loglog} only requires that $\frac{\log n}{\bh_n}$ grows faster than $\log\log n$.

An interesting consequence of Theorem~\ref{thm:cutoff-random} is the following.

\begin{corollary}\label{cor:bounded-degree-graphs}
Let $d\geq 2$ be a fixed integer and let $(G_n=(V_n,E_n))_{n\geq 1}$ be a sequence of $d$-regular digraphs (i.e.\ for all vertices, the in-degree and out-degree are both equal to $d$), with $|V_n|=n$. For each $n\geq 1$, let $\pi_n$ be a uniformly chosen permutation on $V_n$, and $k_n$ be a positive integer. Consider the walk on $G_n$ which moves as a simple random walk for $k_n$ steps and then takes a step according to $\pi_n$. This walk has transition matrix $Q_n=M_n^{k_n}\Pi_n$, where $M_n$ is the transition matrix of the simple random walk on $G_n$, and $\Pi_n$ is the matrix associated to $\pi_n$. Then, if $k_n=o\left(\frac{\log n}{\log\log n}\right)$, the sequence of chains $(Q_n)_{n\geq 1}$ has cutoff at time $\frac{\log n}{{\bf H}_n}$ where
\[
{\bf H}_n=\frac{1}{n}\sum_{u,v\in V(G_n)} M_n^{k_n}(u,v)\log \frac{1}{M_n^{k_n}(u,v)}\, \cdot
\]
\end{corollary}
\begin{proof}[Proof of Corollary~\ref{cor:bounded-degree-graphs}]
It suffices to observe that the bistochastic matrix $P_n=M_n^{k_n}$ satisfies the assumptions of Theorem~\ref{thm:cutoff-random}. Indeed, since $d\geq 2$, the probabilities $M_n^{k_n}(u,v)$ are all at most $1/2$ and assumption~\ref{assum:branching} is satisfied. And since $k_n=o\left(\frac{\log n}{\log\log n}\right)$, for all $u,v\in V_n$ such that $M_n^{k_n}(u,v)>0$, we have
\[
M_n^{k_n}(u,v)\geq \left(\frac{1}{d}\right)^{k_n}=\frac{1}{n^{o(1)}}\, ,
\]
hence assumption~\ref{assum:sparse} is satisfied. Moreover, we have
\[
{\bf H}_n\leq \log\left( d^{k_n}\right)=o\left(\frac{\log n}{\log\log n}\right)\, ,
\]
and assumption~\ref{assum:loglog} is satisfied as well.
\end{proof}

The following proposition (whose proof is given in Section~\ref{subsec:reverse}) shows that the threshold $k_n=o\left(\frac{\log n}{\log\log n}\right)$ is not far from being optimal.

\begin{proposition}\label{prop:bounded-degree-graphs-reverse}
In the setting of Corollary~\ref{cor:bounded-degree-graphs}, for $k_n=\Omega(\log n)$, one may find examples of graphs for which the chain $Q_n$ does not have cutoff.
\end{proposition}

Our second result concerns deterministic permutations. Let us first recall the result of \citet*{chatterjee2020speeding}. For $A\subset \Omega_n$, let $\cE_n(A)$ be the subset of states attainable by $P_n$ in one step starting from $A$. Assume that the sequence of permutations $(\pi_n)_{n\geq 1}$ is such that for all $n\geq 1$, there exists $\alpha_n\in (0,1)$ such that for all $A\subset \Omega_n$ with $|A|\leq n/2$, we have
\begin{equation}\label{assum:expansion}
|\cE_n\circ \pi_n\circ \cE_n(A)|\geq (1+\alpha_n)|A|\, .
\end{equation}
Under some assumptions on $P_n$, \citet*{chatterjee2020speeding} show that the total-variation distance of $Q_n$ satisfies
\[
\cD^{(n)}\left(\frac{(1+o(1))4\log n}{\alpha_n^2\delta_n^8}\right)\cv 0\, ,
\]
where $\delta_n=\min\{P_n(x,y),\, x,y\in\Omega_n,\, P_n(x,y)>0\}$. We show that the dependence on $\delta_n$ can be improved from $\delta_n^{-8}$ to $\delta_n^{-4}$, under weaker assumptions on $P_n$.

\begin{theorem}\label{thm:deterministic}
Let $P_n$ be a bistochastic transition matrix on $\Omega_n$ with $|\Omega_n|=n$, and assume that there exists $\gamma>0$ such that for all $n$ and for all $x\in\Omega_n$, we have $P_n(x,x)\geq\gamma$. Let $\pi_n$ be a permutation on $\Omega_n$ satisfying assumption~\eqref{assum:expansion}. Then the total-variation distance of $Q_n=P_n\Pi_n$ satisfies
\[
\cD^{(n)}\left(\frac{17\log n}{\alpha_n^2\delta_n^4}\right)\cv 0\, .
\]
\end{theorem}

\subsection{Related work}

Theorem~\ref{thm:cutoff-random} describes a sharp transition in the convergence to equilibrium: the first order of the mixing time $\tmix(\varepsilon)$ does not depend on the target $\varepsilon$. In other words, the total-variation distance to the stationary measure abruptly falls from $1$ to $0$ over a short time-scale which is negligible with respect to the mixing time. This behavior is known as cutoff. Since its discovery in the early 1980's by \citet*{diaconis1981generating} and \citet*{aldous1983mixing}, this phenomenon has been shown to occur for a wide variety of finite Markov chains. However, the question of characterizing the chains which have this property is still largely open, except in some specific settings such as birth-and-death chains~\citep*{ding2010total}, and more generally random walks on trees~\citep*{basu2015characterization}. Recently, an impressive step was made by \citet*{salez2021cutoff}, who showed that all non-negatively curved Markov chains satisfying a simple condition have cutoff.

In recent years, a lot of attention has been devoted to random instances: instead of being fixed, the sequence of transition matrices is drawn at random according to some probability distribution, and one looks for a result that holds for almost all realizations. A recurring observation in this framework is that random instances often exhibit cutoff, at a time that can be described in terms of entropy (the ``entropic time''). For random walks on random graphs, this was pioneered by \citet*{lubetzky2010cutoff}, who established cutoff for the simple and non-backtracking random walks on random regular graphs. This was later extended to random walks on random graphs with given degrees~\citep*{ben-salez,berestycki2018random,10.1214/18-AIHP911}, and to random directed graphs~\citep*{bordenave2018random}. Other notable examples are random lifts of graphs~\citep*{bordenave_lacoin_2021,conchon2019cutoff}, random Cayley graphs for Abelian groups~\citep*{hermon2021cutoff-abelian}, and for upper triangular matrices~\citep*{hermon2021cutoff-triang}.

\citet*{bordenave2019cutoff} considered random transition matrices obtained by permuting independently and uniformly at random each row of a fixed transition matrix $P$. They establish cutoff at the entropic time $\frac{\log n}{\bh}$, with $\bh$ given in~\eqref{eq:def-h}. Our setup is similar, except the perturbation is reduced from $n$ to only one random permutation. We note however that our model is significantly simplified by assuming that $P$ has uniform stationary distribution. Also, their result holds under weaker versions of assumptions~\eqref{assum:sparse} and~\eqref{assum:branching}. On the other hand, they assume that $\bh$ is of order $1$, whereas we only need $\bh=o\left(\frac{\log n}{\log\log n}\right)$.

Another model that has some resemblance with our setup is the one considered by \citet*{hermon2020universality}. The authors consider a random graph formed by adding to fixed bounded-degree graph the edges of a uniform random matching on the vertex set. The random walk is shown to have cutoff, at a time that can be described through the entropy of the walk on an auxiliary graph referred to as the quasi-tree. Similarly to our result, this gives a simple random perturbation that makes a large class of transition matrices have cutoff. Let us note that one main difference is that our model has a spontaneous non-backtracking tendency which simplifies the analysis of typical paths.

The spectral properties of the random matrix $Q=P\Pi$ were investigated by \citet*{bordenave2020spectral}. In particular, they show that, under some sparsity and regularity assumptions on $P$, the spectral gap of $Q$ is bounded away from $0$, which implies that the mixing time is $O(\log n)$.

Finally, let us note that multiplying a transition matrix by a permutation fits into the general topic of speeding up Markov chains. For references on this topic, we refer the reader to the paper of \citet*{chatterjee2020speeding} and to the survey of \citet*{hildebrand2005survey} for the case of random walks on groups.

\subsection{Organization of the paper}

Theorem~\ref{thm:cutoff-random} and Proposition~\ref{prop:bounded-degree-graphs-reverse} are proven in Section~\ref{sec:random}, and Theorem~\ref{thm:deterministic} in Section~\ref{sec:deterministic}.

\section{Cutoff for random permutations}\label{sec:random}

To alleviate notation, the dependence on $n$ will often be implicit: for instance, we will write $\Omega$ instead of $\Omega_n$ and $P$ instead of $P_n$. Also, $\delta$ will stand for $\delta_n$ as defined in~\eqref{assum:sparse}, and $\bh$ for $\bh_n$.

\subsection{Coupling with i.i.d.\ samples}\label{sec:coupling}

Before entering into the proof of Theorem~\ref{thm:cutoff-random}, we describe a coupling for typical trajectories, which helps approximating the annealed law of the walk and will be crucially used later on. This coupling takes advantage of the fact that the walk and the permutation $\pi$ along its trajectory can be generated simultaneously as follows: initially, $X_0=x\in\Omega$ and $\Dom(\pi)=\Ran(\pi)=\emptyset$; then at each time $k\geq 0$,
\begin{enumerate}
\item choose an element $Y_{k}\in\Omega$ according to the distribution $P(X_k,\cdot)$;
\item if $Y_{k}\not\in\Dom(\pi)$, i.e.\ if the image of $Y_{k}$ by $\pi$ has not been chosen yet, choose $\pi(Y_{k})$ uniformly at random in $\Omega\setminus\Ran(\pi)$, and update $\Dom(\pi)$ and $\Ran(\pi)$ accordingly; otherwise, $\pi(Y_{k})$ is already defined, and no new image is chosen;
\item in both cases, let $X_{k+1}=\pi(Y_{k})$.
\end{enumerate}
 The sequence $(X_k)_{k\geq 0}$ is then exactly distributed according to the annealed law. Now, consider a sequence $(X_k^\star,Y_k^\star)_{k\geq 0}$ generated in the following way: initially $X_0^\star=x\in\cH$; then at each time $k\geq 0$,
\begin{enumerate}
\item choose an element $Y^\star_{k}\in\Omega$ according to the distribution $P(X_k^\star,\cdot)$;
\item let $X^\star_{k+1}$ be uniformly distributed on the entire set $\Omega$.
\end{enumerate}

Note that the process $(X_k^\star,Y^\star_{k})_{k\geq 1}$ is simply an i.i.d.\ sequence and that both $X_k^\star$ and $Y_{k}^\star$ are uniformly distributed on $\Omega$. The two processes $(X_k,Y_{k})_{k\geq 0}$ and $(X_k^\star,Y^\star_{k})_{k\geq 0}$ may be coupled until time
\[
T=\inf\{k\geq 0, \, X_k^\star\in\Ran(\pi)\; \text{ or}\; Y_{k}^\star\in\Dom(\pi)\}\, .
\]
Since at step $k$, both $|\Dom(\pi)|$ and $|\Ran(\pi)|$ are at most $k$, we have
\begin{equation}
\label{eq:coupling}
\P\left(T\leq t\right)\leq \frac{2t^2}{n} \, \cdot
\end{equation}

\subsection{Lower bound}

Let $x\in \Omega$ be a fixed starting point and let
$$
t=\left\lceil c\frac{\log n}{\bh}\right\rceil\, ,
$$
for $0<c<1$. Let $A$ be the set of $y\in\Omega$ such that there exists a path from $x$ to $y$ which has probability larger than $\frac{\log n}{n}$ to be taken by the chain in $t$ steps. Since, for all $y\in A$, we have $Q^t(x,y)\geq \frac{\log n}{n}$, and since $Q^t(x,\cdot)$ is a probability, the set $A$ has size less than $\frac{n}{\log n}$. Hence
$$
\cD_x(t)\geq Q^t(x,A)-\frac{|A|}{n}\geq Q^t(x,A)-\frac{1}{\log n}\, \cdot
$$
Taking expectation with respect to the permutation, we have
$$
\E Q^t(x,A)\geq \P_x\left(\prod_{s=0}^{t-1}P(X_{s},Y_s)>\frac{\log n}{n}\right)=\P_x\left(\prod_{s=0}^{t-1}P(X^\star_{s},Y^\star_s)>\frac{\log n}{n}\right)+o(1)\, ,
$$
where the last equality is by~\eqref{eq:coupling}. Taking logarithm in the complementary event, singling out the particular first step, and using Chebychev Inequality, we have
\begin{align*}
\P_x\left(\prod_{s=0}^{t-1}P(X^\star_{s},Y^\star_s)\leq \frac{\log n}{n}\right)
&= \P_x\left(\sum_{s=1}^{t-1} \log \frac{1}{P(X^\star_{s},Y^\star_s)}  \geq \log n +o(\log n)\right)\\
&=O\left(\frac{\bsigma^2}{(1-c)^2\bh \log n}\right)\, ,
\end{align*}
where
\[
\bsigma^2=\frac{1}{n}\sum_{x,y\in\Omega} P(x,y)\left( \log\frac{1}{P(x,y)}-\bh\right)^2\, .
\]
By assumption~\eqref{assum:sparse}, we have
\begin{equation}\label{assum:variance}
\bsigma^2=o\left(\bh \log n\right)\, .
\end{equation}
Hence
\[
\min_{x\in\Omega} \E\left[\cD_x(t)\right] \underset{n\to +\infty}\longrightarrow 1\, ,
\]
which yields $\cD(t)\cvp 1$.
\subsection{Upper bound}
Let
$$
t=2\left\lceil \frac{c\log n}{2\bh}\right\rceil \, ,
$$
for $c>1$. Letting $P_\star$ be the matrix defined by
\[
\forall x,y\in\Omega\, ,\; P_\star(x,y)=P(y,x)\, ,
\]
and $Q_\star=P_\star\Pi^{-1}$, a key observation is that for all $x,y\in\Omega$,
\[
Q^t(x,\pi(y))= \sum_{u,v\in\Omega} Q^{t/2}(x,u)Q_\star^{t/2}(y,v)\1_{\{v=\pi^{-1}(u)\}}\, .
\]

The proof of the upper bound then relies on a special exploration procedure, which generates the permutation $\pi$ together with two disjoint weighted trees $\cT_x$ and $\cT_y$, keeping track of only certain paths from $x$ and from $y$. Initially at a time $0$, $\Dom_{0}(\pi)=\Ran_{0}(\pi)=\emptyset$, the tree $\cT_x(0)$ is reduced to $x$ and the tree $\cT_y(0)$ is reduced to $y$. Then, for $k=1,2,\dots$, we iterate the following steps to generate $\cT_x(k)$ and $\cT_y(k)$:
\begin{enumerate}
\item Each vertex $u\in\cT_x(k-1)\cup\cT_y(k-1)$ determines a unique sequence $(u_0,\dots,u_d)$ where $u_0\in\{x,y\}$ is the ancestor of $u$, for each $1\leq i\leq d$, $u_i$ is a child of $u_{i-1}$, and $u_d=u$. The integer $d$ is called the height of $u$, denoted $h(u)$, and the weight of $u$ is defined as
\[
\w(u)=
\begin{cases}
\prod_{s=0}^{d-1} Q(u_s,u_{s+1}) & \text{ if $u_0=x$,}\\
\prod_{s=0}^{d-1} Q_\star(u_s,u_{s+1}) & \text{ if $u_0=y$.}
\end{cases}
\]
If $u_0=x$ (resp.\ $u_0=y$), we say that $(u,i)$ is a stub if $P(u,i)>0$ (resp.\ if $P_\star(u,i)>0$). The weight of a stub is naturally defined as
\[
\w(u,i)=
\begin{cases}
\w(u)P(u,i) & \text{ if $u_0=x$,}\\
\w(u)P_\star(u,i) & \text{ if $u_0=y$.}
\end{cases}
\]
A stub $(u,i)$ is said to be free at time $k-1$, which is denoted $(u,i)\in\cF_{k-1}$, if either $u_0=x$ and $i\not\in\Dom_{k-1}(\pi)$ or $u_0=y$ and $i\not\in\Ran_{k-1}(\pi)$. Pick a free stub $(u,i)\in\cF_{k-1}$ such that
\begin{equation}\label{eq:cond-exploration}
\w(u,i)\geq \wmin=n^{-2/3} \qquad \text{ and } \qquad h(u)<\frac{t}{2}\, \cdot
\end{equation}
If $u_0=x$, set $\Dom_k(\pi)=\Dom_{k-1}(\pi)\cup\{i\}$; if $u_0=y$, set $\Ran_k(\pi)=\Ran_{k-1}(\pi)\cup\{i\}$.
\item If $u_0=x$, then choose $i'=\pi(i)$ uniformly at random in $\Omega\setminus \Ran_{k-1}(\pi)$ and set $\Ran_k(\pi)=\Ran_{k-1}(\pi)\cup\{i'\}$. If $u_0=y$, then choose $i'=\pi^{-1}(i)$ uniformly at random in $\Omega\setminus \Dom_{k-1}(\pi)$ and set $\Dom_k(\pi)=\Dom_{k-1}(\pi)\cup\{i'\}$.
\item If $u_0=x$ and $i'\neq x$, and if for all $j$ with $P(i',j)>0$, we have $j\not\in \Dom_k(\pi)$ and there is no $v\in\cT_x(k-1)$ such that $P(v,j)>0$, then add $i'$ as a child of $u$ in $\cT_x(k-1)$ to form $\cT_x(k)$, and let $\cT_y(k)=\cT_y(k-1)$. If $u_0=y$ and $i'\neq y$, and if for all $j$ with $P_\star(i',j)>0$, we have $j\not\in \Ran_k(\pi)$ and there is no $v\in\cT_y(k-1)$ such that $P_\star(v,j)>0$, then add $i'$ as a child of $u$ in $\cT_y(k-1)$ to form $\cT_y(k)$, and let $\cT_x(k)=\cT_x(k-1)$.
\end{enumerate}

\smallskip

The exploration stage stops when the set of free stubs satisfying~\eqref{eq:cond-exploration} is empty. Let $\tau$ be the number of times step $(2)$ is performed and let $\cT_x=\cT_y(\tau)$ and $\cT_y=\cT_y(\tau)$. Note that, by condition~\eqref{eq:cond-exploration}, we have
\begin{equation}\label{eq:maj-tau}
\wmin\tau \leq \sum_{u\in\cT_x}\w(u)+\sum_{u\in\cT_y}\w(u)\leq t\, .
\end{equation}
Also note that the stub picked in step $(1)$ can be chosen arbitrarily provided it satisfies condition~\eqref{eq:cond-exploration}. We choose to explore in the following order:
\begin{enumerate}
\item[(i)] Reveal around $x$, choosing free stubs with minimal height, until no free stub has height strictly less than
\[
t_0:=\left\lceil \frac{\log\log n}{\log (1/\Delta)}\right\rceil\, .
\]
Call $x_1,\dots,x_m$ the vertices of $\cT_x$ at height $t_0$ at the end of this first stage.
\item[(ii)] Reveal around $y$, choosing free stubs with minimal height, until no free stub has height strictly less than $t_0$. Call $y_1,\dots,y_{m'}$ the vertices of $\cT_y$ at height $t_0$ at the end of this first stage.
\item[(iii)] Sequentially for $\ell=1,\dots,m$, complete the exploration around $x_\ell$, until height $t/2$.
\item[(iv)] Sequentially for $\ell=1,\dots,m'$, complete the exploration around $y_\ell$, until height $t/2$.
\end{enumerate}

\smallskip
Let
\[
r:=\left\lceil \frac{\log n}{12\log(1/\delta)}\right\rceil\, .
\]
By assumption~\eqref{assum:sparse}, we have $r\to +\infty$. An oriented graph $G$ (resp.\ $G_\star$) on $n$ vertices may naturally be associated to the matrix $Q$ (resp.\ $Q_\star$), by placing an oriented edge from $u$ to $v$ when $Q(u,v)>0$ (resp.\ $Q_\star(u,v)>0$). We say that $x$ is a $Q$-root (resp.\ a $Q_\star$-root), which we denote $x\in\cR$ (resp.\ $x\in\cR_\star$), if the (oriented) neighborhood of $x$ in $G$ (resp.\ $G_\star$) up to level $r$ is a tree.

\begin{lemma}\label{lem:reach-root}
For $s=t_0\wedge r$, we have
\[
\max_{x\in\Omega} Q^{s}(x,\cR^c)\overset{\P}\longrightarrow 0\quad \text{ and }\quad \frac{|\cR_\star^c|}{n}\overset{\P}\longrightarrow 0\, .
\]
\end{lemma}
\begin{proof}[Proof of Lemma~\ref{lem:reach-root}]
Let us first prove the first claim. Let $R=\left\lceil \frac{\log n}{5\log(1/\delta)}\right\rceil$. When sequentially revealing the neighborhood of $x$ in $G$ up to level $R$, choosing free stubs with minimal height, the total number of steps is at most $(1/\delta)^{R}=n^{1/5+o(1)}$, and at each step, the probability of having a ``bad event'', i.e.\ choosing a state $u$ such that either $u=x$ or there exists $v$ with $P(u,v)>0$ and $P(u',v)>0$ for some other already exposed $u'$, is at most $\frac{(1/\delta)^{R+2}}{n-(1/\delta)^R}=n^{-4/5+o(1)}$. Hence, the probability of having at least two such bad events up to level $R$ is $n^{-6/5+o(1)}$. Hence, with high probability, all $x\in\Omega$ are such that there is at most one problematic vertex in their $R$-neighborhood. Let $x\in\Omega$. If there is no problematic vertex in its $R$-neighborhood, then the claim is clear. Otherwise, let $u$ be this single problematic vertex. If by time $s$ the walk has exited the only path from $x$ to $u$, then it lies on a $Q$-root. By assumption~\eqref{assum:branching}, at each step, the probability to follow the path is at most $\Delta<1$, except, in the case $u\neq x$, when the walk is at $u$, which can happen at most once in every two steps. Hence the probability not to have exited the path by time $r$ is at most $\Delta^{\frac{s}{2}-1}=o(1)$, which proves the first claim. Since $Q$ has uniform stationary distribution, this implies that $\frac{|\cR^c|}{n}\overset{\P}\longrightarrow 0$. Since $P$ and $P_\star$ satisfy the same assumptions, the same can be said for $|\cR_\star|$, establishing the second claim.
\end{proof}

Let $\cB_x$ be the set of vertices at (oriented) distance at most $r$ from $x$ in $G$ (i.e.\ $\cB_x$ corresponds to the states that can be reached in less than $r$ steps by the chain $Q$). We have $|\cB_x|\leq (1/\delta)^r=o(n)$. Combined with Lemma~\ref{lem:reach-root}, this yields
\[
\cD(t+s)\leq \max_{x\in\cR}\cD_x(t)+o_\P(1)\leq \max_{x\in\cR}\sum_{y\in\cR_\star\setminus \cB_x}\left(\frac{1}{n}-Q^t(x,\pi(y)\right)_++o_\P(1)\, .
\]
The rest of the proof is devoted to establishing 
\begin{equation}\label{eq:to-prove}
\min_{x\in\cR}\min_{y\in\cR_\star\setminus \cB_x}Q^t(x,\pi(y))\geq \frac{1+o_\P(1)}{n}\, \cdot
\end{equation}
Let $\cH_x$ be the set of free stubs $(u,i)$ with $u\in\cT_x$ and $h(u)=\frac{t}{2}-1$ at the end of the exploration stage. Similarly, let $\cH_y$ be the set of free stubs $(v,j)$ with $v\in\cT_y$ and $h(v)=\frac{t}{2}-1$. Let also
\[
\theta=\frac{1}{n(\log n)^2}\, \cdot
\]
We have
\[
Q^t(x,\pi(y))\geq \sum_{(u,i)\in\cH_x}\sum_{(v,j)\in\cH_y}\w(u,i)\w(v,j)\1_{\w(u,i)\w(v,j)\leq \theta}\1_{\{j=\pi(i)\}}\, .
\]
Note that, conditionally on the exploration stage, the sum on the right-hand side above can be written as $\sum_{i\not\in \Dom_\tau(\pi)}a_{i,\pi(i)}$, where the numbers $a_{i,j}$ satisfy
\[
0\leq a_{i,j}\leq \theta \qquad \text{ and } \sum_{\substack{i\not\in \Dom_\tau(\pi)\\ j\not\in \Ran_\tau(\pi)}}a_{i,j}\leq 1\, .
\]
Define
\[
Z_\theta(x,y)= \sum_{(u,i)\in\cH_x}\sum_{(v,j)\in\cH_y}\w(u,i)\w(v,j)\1_{\w(u,i)\w(v,j)\leq \theta}\, ,
\]
and note that $(n-|\Dom_\tau(\pi)|)^{-1}Z_\theta(x,y)$ corresponds to the conditional expectation of $\sum_{i\not\in \Dom_\tau(\pi)}a_{i,\pi(i)}$ given the exploration stage. Using~\citet*[Proposition 1.1, or rather the refined bound for the left tail given in Theorem 1.5, (ii)]{chatterjee2007stein}, we have, for all $\varepsilon>0$,
\begin{align*}
\P\left(nQ^t(x,\pi(y))<Z_\theta(x,y)-\varepsilon\right)
&\leq \P\left( (n-|\Dom_\tau(\pi)|)\sum_{i\not\in \Dom_\tau(\pi)}a_{i,\pi(i)}<Z_\theta(x,y)-\varepsilon\right)\\
&\leq \exp\left(-\frac{\varepsilon^2}{4\theta n}\right)=o\left(\frac{1}{n^2}\right)\, .
\end{align*}
To complete the proof of~\eqref{eq:to-prove}, we now have to show that $\min_{x\in\cR}\min_{y\in\cR_\star\setminus \cB_x} Z_\theta(x,y)=1+o_\P(1)$. To do so, we decompose $1-Z_\theta(x,y)$ into three terms that are treated separately: the weight of free stubs at height $t/2$ but with product larger than $\theta$, the weight of stubs that are not free and childless (i.e.\ those that did not satisfy the condition at step $(3)$ of the exploration stage), and the weight of free stubs at height less than $t/2$ (i.e.\ those that have weight smaller than $\wmin$).

\begin{lemma}\label{lem:large-weights}
For all $\varepsilon>0$,
\[
\P\left(\sum_{(u,i)\in\cH_x}\sum_{(v,j)\in\cH_y}\w(u,i)\w(v,j)\1_{\left\{\w(u,i)\w(v,j)> \theta\right\}}>\varepsilon\right)=o\left(\frac{1}{n^2}\right)\, .
\]
\end{lemma}
\begin{proof}[Proof of Lemma~\ref{lem:large-weights}]
First observe that
\[
\sum_{(u,i)\in\cH_x}\sum_{(v,j)\in\cH_y}\w(u,i)\w(v,j)\1_{\left\{\w(u,i)\w(v,j)> \theta\right\}}\leq \sum_{z\in\{x,y\}} \sum_{(u,i)\in\cH_z}\w(u,i)\1_{\left\{\w(u,i)> \sqrt{\theta}\right\}}\, .
\]
We will show that, for all $\varepsilon>0$,
\[
\P\left(\sum_{(u,i)\in\cH_x}\w(u,i)\1_{\left\{\w(u,i)> \sqrt{\theta}\right\}}>\varepsilon\right)=o\left(\frac{1}{n^2}\right)\, .
\]
Since the argument for $y$ is completely similar, we do not include it (the only difference is that we have to condition also on the whole tree $\cT_x$ but this does not change the argument).

Let $\cJ_0$ be the $\sigma$-field corresponding to steps $(i)$ and $(ii)$, with $x_1,\dots,x_m$ the vertices at height $t_0$. For $\ell=1,\dots, m$, let $\cJ_\ell$ be the $\sigma$-field containing $\cJ_0$ plus the first $\ell$  steps of step $(iii)$, i.e.\ the subtrees rooted at $x_1,\dots,x_\ell$. Observe that
\[
\sum_{(u,i)\in\cH_x}\w(u,i)\1_{\left\{\w(u,i)> \sqrt{\theta}\right\}}\leq \sum_{\ell=1}^m \w(x_\ell)W_\ell\, ,
\]
where
\[
W_\ell=\sum_{\substack{(u,i)\in\cH_x\\ u_r=x_\ell}}\frac{\w(u,i)}{\w(x_\ell)}\ind_{\left\{\frac{\w(u,i)}{\w(x_\ell)}>\sqrt{\theta}\right\}}\, .
\]
Now consider the sequence $(M_k)_{k=0}^m$ defined by $M_0=0$, and, for $k=1,\dots,m$,
\[
M_k=\sum_{\ell=1}^k \w(x_\ell)\left(W_\ell-\E[W_\ell\given \cJ_{\ell-1}]\right)\, .
\]
Note that
\[
\E[W_\ell\given \cJ_{\ell-1}]\leq \P_{x_\ell}\left(\prod_{s=0}^{t/2-t_0-1}P(X_s,Y_s)>\sqrt{\theta}\given \cJ_{\ell -1}\right)\, .
\]
Using the coupling of Section~\ref{sec:coupling} and bound~\eqref{eq:maj-tau}, we have
\[
\E[W_\ell\given \cJ_{\ell-1}]\leq O\left( \frac{t^2}{\wmin n}\right)+\P_{x_\ell}\left(\prod_{s=0}^{t/2-t_0-1}P(X^\star_s,Y^\star_s)>\sqrt{\theta}\right)\, .
\]

Singling out the particular first step, taking logarithm, and using that $t_0\bh=o(\log n)$, we have
\[
\P_{x_\ell}\left(\prod_{s=0}^{t/2-t_0-1}P(X^\star_s,Y^\star_s)>\sqrt{\theta}\right)\leq \P\left(\sum_{s=1}^{t/2-t_0-1} \left(\bh -\log \frac{1}{P(X^\star_s,Y^\star_s)}\right)>u_n\right)\, ,
\]
with $u_n=\frac{(c-1)}{2}\log n +o(\log n)$. Using Chebychev Inequality, we get $\E[W_\ell\given \cJ_{\ell-1}]=o(1)$, uniformly in $\ell$, which implies
\[
\sum_{\ell=1}^m \w(x_\ell)W_\ell\leq M_m+o(1)\, .
\]
On the other hand, $|M_k-M_{k-1}| \leq \w(x_k)\leq \Delta^{t_0}\leq\frac{1}{\log n}$, and
\[
\sum_{k=1}^m\E\left[ (M_k-M_{k-1})^2\given \cJ_{k-1}\right]\leq \sum_{k=1}^m\w(x_k)^2\E[W_k\given \cJ_{k-1}]=o(\Delta^{t_0})\, .
\]
Hence, by Freedman Inequality~\citep*{freedman1975tail}, we obtain
\[
\P\left(M_m>\varepsilon\right)\leq o(1)^{\varepsilon \log n}=o\left(\frac{1}{n^2}\right)\, . \qedhere
\]

\end{proof}

Let $\cF$ be the set of free stubs $(u,i)$ with $u\in\cT_x\cup\cT_y$ at the end of the exploration stage.

\begin{lemma}\label{lem:non-free}
For all $\varepsilon>0$,
\[
\P\left(x\in\cR,\, y\in\cR_\star\setminus\cB_x,\, \sum_{(u,i)\in\cF} \w(u,i) <2-\varepsilon\right)=o\left(\frac{1}{n^2}\right)\, .
\]
\end{lemma}
\begin{proof}[Proof of Lemma~\ref{lem:non-free}]
For $k\geq 0$, let $\cF_k$ denote the set of free stubs $(u,i)$ with $u\in\cT_x(k\wedge\tau)\cup\cT_y(k\wedge\tau)$ (i.e.\ after $k\wedge \tau$ rounds of the exploration process), and consider the random variable
$$W_k:=\sum_{(u,i)\in\cF_k}\w(u,i).$$
Initially, $W_0=2$, and, for $k\geq 1$, the difference $W_k-W_{k-1}$ is either $0$ or negative, depending on whether the condition of step~$(3)$ is satisfied or not. More precisely, denoting by $(u_k,i_k)$ the free stub selected at the $k^\textrm{th}$ iteration of step~$(1)$ and letting $i_k'$ be the state selected at step~$(2)$, we have
for all $k\geq 1$,
\begin{align*}
W_k-W_{k-1} &=  -\1_{k\leq\tau} \w(u_k,i_k)\Bigg\{\1_{u_k\in\cT_x(k-1)}\1_{i_k'\in\{x\}\cup\cN_{\rm in}\left(\Dom_k(\pi)\right)\cup\cN_{\rm in}\left(\cN_{\rm out}\left(\cT_x(k-1)\right)\right)}\\
&\hspace{3cm}+\1_{u_k\in\cT_y(k-1)}\1_{i_k'\in\{y\}\cup\cN_{\rm out}\left(\Ran_k(\pi)\right)\cup\cN_{\rm out}\left(\cN_{\rm in}\left(\cT_y(k-1)\right)\right)}\Bigg\}\, ,
\end{align*}
where for $A\subset\Omega$,
\[
\cN_{\rm in}(A)=\{z\in\Omega,\, P(z,A)>0\}\qquad\text{ and }\qquad \cN_{\rm out}(A)=\{z\in\Omega,\,\exists a\in A,\, P(a,z)>0\}\, .
\]
Now, let $\{\cG_k\}_{k\geq 0}$ be the natural filtration associated with the exploration stage. Note that $\tau$ is a stopping time, that $\w(u_k,i_k)$ is $\cG_{k-1}$-measurable and that, if $u_k\in\cT_x(k-1)$, then
\begin{align*}
&\P\left(i_k'\in \{x\}\cup\cN_{\rm in}\left(\Dom_k(\pi)\right)\cup\cN_{\rm in}\left(\cN_{\rm out}\left(\cT_x(k-1)\right)\right)\given \cG_{k-1}\right)\\
 &\quad \leq \frac{1}{n-|\Ran_{k-1}(\pi)|}\left(1+\frac{1}{\delta}|\Dom_k(\pi)|+\frac{1}{\delta^2}|\cT_x(k-1)|\right)\\
&\quad \leq \frac{3k}{\delta^2(n-k)}
\end{align*}
where we used that $|\Ran_{k-1}(\pi)|$, $|\Dom_k(\pi)|$, and $|\cT_x(k-1)|$ are at most $k$. Similarly, if $u_k\in\cT_y(k-1)$, then
\begin{align*}
\P\left(i_k'\in \{y\}\cup\cN_{\rm out}\left(\Ran_k(\pi)\right)\cup\cN_{\rm out}\left(\cN_{\rm in}\left(\cT_y(k-1)\right)\right)\given \cG_{k-1}\right) &\leq \frac{3k}{\delta^2(n-k)}\, \cdot
\end{align*}
Hence,
\[
\E\left[ W_k-W_{k-1}\given \cG_{k-1}\right]\geq -\ind_{k\leq\tau}\w(u_k,i_k) \frac{3k}{\delta^2(n-k)}\, ,
\]
and
\[
\E\left[ (W_k-W_{k-1})^2\given \cG_{k-1}\right]\leq  \ind_{k\leq\tau}\w(u_k,i_k)^2 \frac{3k}{\delta^2(n-k)} \, \, \cdot
\]
Using that
\[
\wmin\tau \leq \sum_{k\leq\tau}\w(u_k,i_k)\leq t\, ,
\]
and assumption~\eqref{assum:sparse}, we get
\[
m:=\sum_{k=1}^\tau \E\left[ W_{k-1}-W_k\given \cG_{k-1}\right]\leq \frac{n^{o(1)}t^2}{n\wmin-t}\, ,
\]
and
\[
v:=\sum_{k=1}^\tau \E\left[ (W_k-W_{k-1})^2\given \cG_{k-1}\right]\leq \frac{n^{o(1)}t^2}{n\wmin-t}\, \cdot
\]
Now, fix $\varepsilon>0$ and consider the martingale $(M_k)_{k\geq 0}$ defined by  $M_0=0$ and
$$M_k:=\sum_{i=1}^k\left\{(W_{i-1}-W_i)\wedge\varepsilon-\E\left[(W_{i-1}-W_{i})\wedge\varepsilon \big|\cG_{i-1}\right]\right\}.$$
The increments of $(M_k)_{k\geq 0}$ are bounded by $\varepsilon$ by construction, and the quadratic variation up to time $\tau$ satisfies
\[
\sum_{k=1}^{\tau}\E\left[\left(M_k-M_{k-1}\right)^2\big|\cG_{k-1}\right] \leq v= n^{-\frac{1}{3}+o(1)}\, .
\]
Thus, Freedman Inequality~\citep*{freedman1975tail} yields
\begin{equation}
\label{eq:freedman}
\P\left(M_\tau> 7\varepsilon\right) \leq  \left(\frac{ev}{v+7\varepsilon^2}\right)^{7}=o\left(\frac{1}{n^2}\right)\, .
\end{equation}
Now on the event $x\in\cR$ and $y\in\cR_\star\setminus\cB_x$, one must have $\max_{k}(W_{k-1}-W_k)\leq \varepsilon$ for $n$ large enough (since $\Delta<1$ and $r\to +\infty$), hence
\[
\{x\in\cR,y\in\cR_\star\setminus\cB_x\} \subseteq \left\{W_0-W_\tau\leq M_\tau+m\right\}\, .
\]
Since $W_0-W_\tau=2-\sum_{(u,i)\in\cF}\w(u,i)$ and $m=o(1)$, this concludes the proof.
\end{proof}

\begin{lemma}\label{lem:small-weights}
For all $\varepsilon>0$,
\[
\P\left(\sum_{(u,i)\in\cF} \w(u,i)\1_{\left\{\w(u,i)<\wmin\right\}} >\varepsilon\right)=o\left(\frac{1}{n^2}\right)\, .
\]
\end{lemma}
\begin{proof}[Proof of Lemma~\ref{lem:small-weights}]
Let $\cF_x$ be the set of free stubs $(u,i)$ with $u\in\cT_x$. Consider $m=\lfloor\log n\rfloor$ independent Markov chains with transition matrix $Q$ all started at $x$, each being killed as soon as its weight falls below $\wmin$, and write $A$ for the event that their trajectories form a tree of height less than $t/2$. The conditional probability of $A$ given the permutation $\pi$ satisfies
\[
\P\left(A| \pi\right)\geq \left(\sum_{(u,i)\in\cF_x} \w(u,i)\1_{\left\{\w(u,i)<\wmin\right\}}\right)^m\, .
\]
Taking expectation and using Markov inequality, we deduce that
\[
\P\left(\sum_{(u,i)\in\cF_x} \w(u,i)\1_{\left\{\w(u,i)<\wmin\right\}} >\varepsilon\right) \leq  \frac{\P(A)}{\varepsilon^m}\, ,
\]
where the average is now taken over both the Markov chains and the permutation. We will now show that $\P(A)=o(1)^m$, so that $\varepsilon^{-m}\P(A)=o(n^{-2})$. Generate the $m$ killed trajectories one after the other, revealing the permutation along the way, as described in Section~\ref{sec:coupling}. Given that the first $\ell-1$ trajectories form a tree of height less than $t/2$, we claim that the conditional chance that the $\ell^\textrm{th}$ one also fulfills the requirement is $o(1)$, uniformly in $1\leq\ell\leq m$. Indeed,
\begin{itemize}
\item either its weight falls below $\eta=(1/{\log n})^{2}$ before it ever leaves the graph spanned by the first $\ell-1$ trajectories and reaches an unvisited state: thanks to the tree structure, there are at most $\ell-1< m$ possible trajectories to follow, each having weight at most $\eta$, so the chance is less than $m\eta=o(1)$.
\item or the remainder of its trajectory after the first unpaired half-edge has weight less than ${\wmin}/{\eta\delta}$ (where $\delta$ accounts for the step exiting the graph spanned by the previous trajectories): this part consists of at most $t/2$ states which can be coupled with uniform samples from $\Omega$ for a total-variation cost of $\frac{mt^2}{2n}$, as in Section~\ref{sec:coupling}. Thus, the conditional chance is at most
\[
\frac{mt^2}{2n}+\P\left(\prod_{s=1}^{t/2}P(X_{s}^\star,Y_s^\star) <  \frac{\wmin}{\eta\delta}\right)
 =  o(1)\, ,
\]
by Chebychev Inequality.
\end{itemize}
Since $P_\star$ satisfies the same assumptions as $P$, the same argument may be repeated to show that
\[
\P\left(\sum_{(u,i)\in\cF_y} \w(u,i)\1_{\left\{\w(u,i)<\wmin\right\}} >\varepsilon\right)=o\left(\frac{1}{n^2}\right)\, . \qedhere
\]
\end{proof}

Combining Lemma~\ref{lem:large-weights}, \ref{lem:non-free} and \ref{lem:small-weights}, we obtain $\min_{x\in\cR}\min_{y\in\cR_\star\setminus \cB_x} Z_\theta(x,y)=1+o_\P(1)$, which concludes the proof of the upper bound.

\subsection{Proof of Proposition~\ref{prop:bounded-degree-graphs-reverse}}\label{subsec:reverse}

In the setting of Corollary~\ref{cor:bounded-degree-graphs}, let us show that, for $k_n=\Omega(\log n)$, there exist graphs for which the chain $Q_n$ does not have cutoff. Consider a graph $G_n$ formed by the disconnected union of one $3$-regular Ramanujan graph $R_n$ of size $r_n\sim \frac{n}{\sqrt{\log n}}$, and of $\frac{n-r_n}{4}$ copies of $K_4$ (the complete graph on four vertices). Take $k_n\geq a\log n$ with $a>\frac{3}{\log(2)}$, and consider the Markov chain with transition matrix $Q_n=M_n^{k_n}\Pi_n$ where $M_n$ is the transition matrix of the simple random walk on $G_n$ and $\Pi_n$ is a uniform random permutation matrix. We claim that, starting from one copy of $K_4$, the mixing time of $Q_n$ is determined by the hitting time of $R_n$. Indeed, observe that, when restricted to $G_n\setminus R_n$, the chain $Q_n$ is very similar to a random walk on a $4$-regular directed random graph, which takes order $\log n$ steps to mix. Since the hitting time of $R_n$ is approximately distributed as a Geometric random variable with parameter $\frac{1}{\sqrt{\log n}}$, the chain $Q_n$ reaches $R_n$ before it mixes on $G_n\setminus R_n$. Once it has reached $R_n$, then $k_n$ simple random walk steps makes it basically uniformly distributed on $R_n$. This is because, by~\citep*{lubetzky2016cutoff}, the simple random walk on $R_n$ has cutoff at time $\frac{3}{\log(2)}\log n$. Hence, applying the random permutation makes the chain approximately uniformly distributed over a set of $r_n$ vertices, among which a vanishing proportion belong to $R_n$. Now, starting from the uniform distribution over a set of about $r_n$ vertices, the chain on $G_n\setminus R_n$ only requires order $\log\log n$ steps to mix. All in all, the mixing time is of order $\sqrt{\log n}$, and there is no cutoff.

\section{Deterministic permutations}\label{sec:deterministic}

Let $P$ be a bistochastic trnasition matrix on $\Omega$, with $|\Omega|=n$, and let
\[
\delta=\min\left\{ P(x,y) ,\; x,y\in\Omega , P(x,y)>0\right\}\, .
\]
Assume that the matrix $P$ has laziness parameter $\gamma>0$, i.e.\
\[
\forall x\in\Omega\, ,\; P(x,x)\geq \gamma\, .
\]
Let $\pi$ be a fixed permutation on $\Omega$. For $A\subset\Omega$, we denote by $\cE(A)$ the subset of elements attainable in one step of $P$ starting from $A$, and we assume that there exists $\alpha\in (0,1)$ such that for all $A\subset \Omega$ with $|A|\leq n/2$,
\begin{equation}\label{assum:expansion2}
\left| \cE\circ \pi\circ \cE(A)\right|\geq (1+\alpha)|A|\, .
\end{equation}

Consider the matrix $Q=P\Pi$, where $\Pi$ is the matrix associated to $\pi$. We proceed by bounding the mixing time of $Q^2$. Note that by assumption~\eqref{assum:expansion2}, the matrix $Q^2$ is irreducible. For $x,y\in\Omega$, we have
\[
Q^2(x,y)=\sum_{z\in\Omega} P(x, \pi^{-1}(z))P(z,\pi^{-1}(y))\, .
\]
Following \citet*{morris2005evolving}, we define, for $S\subset\Omega$, $S\neq\emptyset$,
\begin{align*}
\varphi_S=\frac{1}{2|S|}\sum_{y\in\Omega} \min\left\{ \sum_{x\in S} Q^2(x,y)\, ,\, \sum_{x\in S^c} Q^2(x,y)\right\}\, .
\end{align*}
Note that, if $y\in \pi\circ\pi(S)$, then for $x=\pi^{-1}\circ \pi^{-1}(y)\in S$, we have
\[
Q^2(x,y)\geq P( x, x)P(\pi(x),\pi(x))\geq \gamma^2\, .
\]
Moreover, since $Q^2$ has uniform stationary distribution, $\sum_{x\in S^c} Q^2(x,y)\leq 1$. Hence,
\[
\sum_{y\in\pi\circ\pi(S)} \min\left\{ \sum_{x\in S} Q^2(x,y)\, ,\, \sum_{x\in S^c} Q^2(x,y)\right\}\geq \gamma^2  \sum_{x\in S^c} Q^2(x,\pi\circ\pi(S))\, .
\]
Similarly,
\[
\sum_{y\in\pi\circ\pi(S^c)} \min\left\{ \sum_{x\in S} Q^2(x,y)\, ,\, \sum_{x\in S^c} Q^2(x,y)\right\}\geq \gamma^2  \sum_{x\in S} Q^2(x,\pi\circ\pi(S^c))\, .
\]
Therefore,
\begin{align*}
\varphi_S &\geq \frac{\gamma^2}{2|S|}\left\{ \sum_{x\in S} Q^2(x,\pi\circ\pi(S^c))+ \sum_{x\in S^c} Q^2(x,\pi\circ\pi(S))\right\}\\
&\geq \frac{\gamma^2\delta^2}{2|S|}\left\{ \left| \pi^{-1}\circ \cE\circ\pi\circ\cE(S)\setminus S\right| + \left|\pi^{-1}\circ \cE\circ\pi\circ\cE(S^c)\setminus S^c\right|\right\}\, .
\end{align*}
Assume without loss of generality that $|S|\leq n/2$. Since $P$ has some laziness, $S\subset \pi^{-1}\circ \cE\circ\pi\circ\cE(S)$, and, by assumption, we have
\begin{align*}
\left| \pi^{-1}\circ \cE\circ\pi\circ\cE(S)\right|&= \left|\cE\circ\pi\circ\cE(S)\right|\geq (1+\alpha)|S|\, .
\end{align*}
We obtain
\[
\varphi_\star:=\min_{S\subset \Omega,\, S\neq \emptyset} \varphi_S \geq \frac{\gamma^2\delta^2\alpha}{2}\, .
\]
By~\citep*[Theorem 4 and Lemma 10]{morris2005evolving}, this implies that the mixing time of $Q$ is bounded by
\begin{align*}
\tmix(\varepsilon)&\leq \int_{4/n}^{1/\varepsilon^2} \frac{4}{u\varphi_\star^2} du \\
& \leq \frac{16}{\gamma^4\delta^4\alpha^2}\left(\log\left(\frac{n}{4}\right)+2\log\left(\frac{1}{\varepsilon}\right)\right)\, .
\end{align*}
Denoting by $\cD^{(\gamma)}(\cdot)$ the worst case total-variation distance when $P$ has laziness $\gamma$, we have in particular that for all $\gamma$ such that $16< 17\gamma^4$,
\[
\cD^{(\gamma)}\left( \frac{17}{\delta^4\alpha^2}\log n\right) \cv 0\, .
\]
Using~\citet*[Proposition 3.1]{chen2013comparison}, we may then extend this bound to all positive laziness parameters: for all $\gamma\in (0,1)$, we have
\[
\cD^{(\gamma)}\left( \frac{17}{\delta^4\alpha^2}\log n\right) \cv 0\, .
\]

\medskip

\noindent{\bf Acknowledgments.}\, We thank Charles Bordenave and Persi Diaconis for helpful discussions. We are also very grateful to Sam Olesker--Taylor for his careful reading and for pointing out to us confusing typos in a previous version.

\bibliographystyle{abbrvnat}
\bibliography{biblio}

\end{document}